\documentclass[11pt,reqno]{amsart}
\usepackage{amssymb,amsmath}
\usepackage[bbgreekl]{mathbbol}
\usepackage[toc,page]{appendix}
\usepackage{mathrsfs}
\usepackage[normalem]{ulem} 

\usepackage{xcolor}
\usepackage[all]{xy}
\newcommand{\aut}{{\rm Aut}}

\newcommand{\Id}{{\rm Id}}

\newcommand{\car}{{\rm char}}

\newcommand{\spe}{\operatorname{Spec}}

\newcommand{\A}{{\mathbb A}}

\newcommand{\K}{{\Bbbk}}

\newcommand{\Q}{{\mathbb Q}}
\newcommand{\Z}{{\mathbb Z}}

\newcommand{\kk}{{\mathbb k}}
\newcommand{\ack}{\overline{{\mathbb k}}}

\newcommand{\pol}{{\K[x,y]}}

\newcommand{\orb}{{\rm O}}
\newcommand{\ophi}{{\rm O}_\varphi}

\newcommand{\acI}{\mathcal I_{\ack}}

\newcommand{\gal}{{\rm Gal}}

\newtheorem{thm}{Theorem}[section]
\newtheorem{pro}[thm]{Proposition}
\newtheorem{cor}[thm]{Corollary}
\newtheorem{lem}[thm]{Lemma}

\theoremstyle{definition}
\newtheorem{defi}[thm]{Definition}

\newtheorem{rem}[thm]{Remark}
\newtheorem{nota}[thm]{Notation}

\newtheorem{exa}[thm]{Example}

\usepackage{enumitem}

\newcommand{\beginenumit}{\begin{enumerate}[leftmargin=\parindent,align=left,labelwidth=\parindent,parsep=8pt,labelsep=0pt,label=\textbullet\ ]}
\newcommand{\beginenum}{\begin{enumerate}[leftmargin=\parindent,align=left,labelwidth=\parindent,parsep=8pt,labelsep=0pt,label=(\arabic*)\ ]}
\newcommand{\beginenumr}{\begin{enumerate}[leftmargin=\parindent,align=left,labelwidth=\parindent,labelsep=0pt,parsep=8pt,label=(\roman*)\ ]}
\newcommand{\beginenuma}{\begin{enumerate}[leftmargin=\parindent,align=left,labelwidth=\parindent,labelsep=0pt,parsep=8pt,label=(\alph*)\ ]}

\begin{document}

\title[]{On the orbits of plane automorphisms and their stabilizers}
\maketitle
\begin{center}
  {\sc Alvaro Rittatore}
  \ and
 {\sc Iván Pan}\footnote{Research of both authors was partially
   supported by CSIC (Udelar), ANII and PEDECIBA, of Uruguay.}
\end{center}

 \begin{abstract}
Let $\Bbbk$ be a perfect field with algebraic closure $\overline{\Bbbk}$. If $H$ is a subgroup of plane automorphisms over $\Bbbk$ and $p\in\overline{\Bbbk}^2$ is a point, we describe the subgroup consisting of plane automorphisms which stabilize the orbit of $p$ under $H$, when this orbit has irreducible closure in $\overline{\Bbbk}^2$. As an application, we treat the case where $H$ is cyclic and the closure of the orbit of $p$ is an arbitrary (non-necessarily irreducible) curve.  
\end{abstract}

\section{Introduction}
 
In \cite{BlSt} the authors study the group of automorphisms
of the affine plane preserving some given curve over an arbitrary
field $\Bbbk$. As an application of their main results, they propose a
classification of the group of automorphisms stabilizing an arbitrary
set $\Delta\subset \ack^2$ --- here $\ack$ denotes the algebraic
closure of $\Bbbk$. However, this classification (see
  \cite[Prop. 3.11]{BlSt}) is
  incomplete, since the authors assume implicitly that
  $\Delta$ is such that the Zariski closure of
$\Delta$ in $\ack^2$  coincides with the zero
locus of $I_\kk(\Delta)$,   the ideal of the $\K$--polynomials
vanishing at $\Delta$. 

Inspired by the aim of completing the  classification
  mentioned above, we have been led to consider the following problem,
which is interesting in its own right and whose study is the objective
 of  this
paper:

\begin{quote}If $H$ is a subgroup of plane
automorphisms over a perfect field $\kk$ and $p\in\ack^2$, then
describe the stabilizer of its orbit, i.e. the set of plane
automorphisms which leave invariant $\orb_H(p)=\bigl\{h(p): h\in H\bigr\}$. 
\end{quote}

 In particular, we show  that when $\Delta$ is
  not closed in $\ack^2$ an interesting phenomenon occurs: in many cases the  group
  of automorphisms stabilizing $\Delta$  is infinite countable and
  therefore not algebraic. 

In Section \ref{sec_prelim} we present the basic definitions and
preliminary results concerning the automorphisms group of a
subset of the affine plane over $\ack$.

In Section \ref{sec:orbitsubgruop} we present our main result, namely, we describe the  stabilizer of  $\orb_H(p)$ when it has irreducible closure in $\ack^2$, see Theorem \ref{thm:autorbit}. As an application, in Section \ref{sec:orbitautom} we consider the case where $H$ is a cyclic group and the closure of $\orb_H(p)$ is a (non-necessarily irreducible) curve, see Theorem \ref{thm_2}.

  \section{Preliminaries}
\label{sec_prelim}

In this section we  introduce the definition of the automorphisms group
of a subset of the affine plane $\A^2$ and recall the description of 
the automorphisms group  of a curve in $\A^2$.
  
\begin{nota} 
The following  notations  will be kept through the paper.
  \beginenum
\item In what follows $\Bbbk$ is an arbitrary  perfect field and  $\ack$ is an
  algebraic closure of $\kk$.

\item We let $\A^2=\spe\bigl(\kk[x,y]\bigr)$ and  denote  
  the set of $\kk$-points and  $\ack$-points in $\A^2$ by $\A^2(\kk)$
  and $\A^2(\ack)$ respectively.
Recall that $\A^2(\ack)$ identifies with the $\ack$-points of 
  $\A^2_{\ack}=\spe(\ack)\times_\kk \A^2$.

  \item  If  $\Delta\subset\A^2(\ack)$, we denote the \emph{vanishing
      ideals} of $\Delta$ by $\mathcal
    I(\Delta)\subset \kk[x,y]$ and $\acI(\Delta)\subset \ack[x,y]$. If
$J\subset \pol$ is an ideal, we   denote  $\mathcal
V(J)=\spe\bigl(\kk[x,y]/J\bigr)$ viewed as a closed subscheme of $\A^2$, and $\mathcal V_{\ack}(J)=\mathcal
V(J)(\ack)\subset
\A^2_{\ack}(\ack)\cong \A^2(\ack)$  the
  \emph{zero set of $J$}.

\item If $\Delta\subset \A^2(\ack)$, we set
$\widehat{\Delta}=\mathcal V_{\ack}\bigl(I(\Delta)\bigr)\subset
\A^2(\ack)$;  the Zariski
closure of $\Delta$ in $\A^2(\ack)$ is denoted by $\overline{\Delta}$
--- hence, $\overline{\Delta}\subset \widehat{\Delta}\subset  
\A^2(\ack)$.

\item The \emph{automorphisms
  group} of $\A^2$ is  denoted by $\aut(\A^2)$ --- recall that $\aut(\A^2)$ is  the (abstract) group of isomorphisms of $\kk$-schemes $f:\A^2\to \A^2$.

\end{enumerate}
\end{nota}

The following well known result, that gives further insight on the
relationship between     $\overline{\Delta}$ and $\widehat{\Delta}$,  follows from 
Galois descent (see for example \cite[Appendix A.j]{Mi}); we include
its proof for the sake of completeness.

  \begin{lem}\label{lem:pro_gal}
    Consider the canonical
    action of 
    $G=\gal(\ack/\kk)$ on $\A^2(\ack)$. If $\Delta\subset \A^2(\ack)$,
      then $\widehat{\Delta}= G\cdot     \overline{\Delta}$. In
      particular,  
    every irreducible component of $\widehat{\Delta}$ is the image of an
    irreducible component of $\overline{\Delta}$ under an element in
    $G$. 
    \end{lem}
    \begin{proof}
It is clear that $\widehat{\Delta}\subset \A^2(\ack)$ is a
$G$--stable closet subset 
containing $\overline{\Delta}$. Therefore, by Galois descent, 
there exists an unique closed subscheme $X\subset \operatorname{Spec}\bigl(\kk[x,y]\bigr)$ such 
that $\widehat{\Delta}$ identifies with the set of $\ack$-points of
$\operatorname{Spec}(\ack)\times_{\operatorname{Spec}(\kk)} X$.  Under the identification $\A^2(\kk)\subset
\A^2(\ack)$, we have that  $X(\kk)$ is the set of $\ack$-points
 of $\widehat{\Delta}$ fixed by the Galois action and
$\widehat{\Delta}=G\cdot \bigr(X\cap 
\A^2(\kk)\bigr)$. It follows that $X(\kk)= 
\overline{\Delta}\cap \A^2(\kk)$ and  $G\cdot
\overline{\Delta}=\widehat{\Delta}$. The last assertion is an easy
consequence of Galois descent.
\end{proof}

We introduce now the main object of study of this work, namely the
group of automorphisms that stabilize an arbitrary  subset of
$\ack$-points of the
affine plane.

\begin{defi}
Let $a:\aut(\A^2)\times \A^2_{\ack}(\ack)\to \A^2_{\ack}(\ack)$, $a(f,
p)= f(\ack)(p)$,   be the
canonical action of $\aut(\A^2)$ on $\A^2(\ack)$. If $\Delta\subset \A^2(\ack)$, we denote the \emph{stabilizer} of $\Delta$ 
under $a$ by   $\aut(\A^2,\Delta)$. 
\end{defi}

\begin{rem}
   \label{rem:invariances}
   \beginenum
 \item Recall that an automorphism $\varphi\in \aut(\A^2)$ is given by a
   pair $(f,g)\in \kk[x,y]\times \kk[x,y]$, such that the
   corresponding endomorphism of $\A^2_{\ack}$ is an automorphism,
   with inverse  $
(f,g)^{-1}= (h,j) \in \kk[x,y]\times \kk[x,y]$.
Under this identification, $a\bigl((f,g),p\bigr)=\bigl(f(p),g(p)\bigr)$.

\item Let $\Delta\subset\A^2(\ack)$ and    $\varphi\in
\aut(\A^2,\Delta)$. If  $\varphi^*:\ack[x,y]\to \ack[x,y]$ denotes the induced automorphism (which 
is defined over $\kk$),  then $\varphi^*$ stabilizes  the ideals 
$ \mathcal I(\Delta)$ and  $ \acI(\Delta)$.
If follows that $\varphi$ stabilizes $\overline{\Delta}$ and
$\widehat{\Delta}$;  therefore $\varphi^*$ also stabilizes the ideals
$\mathcal I(\widehat{\Delta})$ and and  $\acI(\widehat{\Delta})$.

 In other words: 
  \begin{equation}
    \label{eqinc2}
    \aut(\A^2,\Delta)\subset 
    \aut\bigl(\A^2,\overline{\Delta}\bigr)\subset \aut\bigl(\A^2,\widehat{\Delta}\bigr).
    \end{equation}
\end{enumerate}
  \end{rem}

  \begin{nota} If $H\subset \aut(\A^2)$ is a subgroup
and $p\in \A^2(\ack)$, we denote the \emph{$H$-orbit} of $p$ as $\orb_H(p)$. 

   \end{nota}

\begin{exa}\label{eje1}
It is well known that the inclusions \eqref{eqinc2}  may be
strict. For example if  $\kk=\Q$ 
and $\Delta=\bigl\{(\sqrt{2},0)\bigr\}$, then $\widehat{\Delta}=\bigl\{ (\sqrt{2},0), (-\sqrt{2},0)\bigr\}$
 and 
$\varphi=(-x,y+x^2-2)\in \aut(\A^2,\widehat{\Delta})\setminus \aut(\A^2,\Delta)$.

The fact that the first inclusion may be strict is a key
point in our study of the stabilizers of orbits.

  \end{exa}

   In \cite{BlSt} the authors study the stabilizer of a closed subset
   of the affine plane. In particular, they
  give an explicit description of the geometrically irreducible curves $\mathcal C$
  such that  $\aut\bigl(\A^2,\mathcal C(\ack)\bigr)$ is an algebraic
  group. We briefly recall now these results --- according to Remark \ref{rem:invariances}, we present an
  element of $\aut(\A^2)$ as a pair of polynomials $(f,g)\in \kk[x,y]$.

\begin{rem} If  $\Delta\subset \A^2({\ack})$ is finite, then
  $\aut(\A^2,\Delta)$ is not an algebraic group (see
  \cite[Proposition 3.11]{BlSt}) --- notice that 
  if   $\Delta=\bigl\{(a_i,b_i)\in\A^2(\ack); i=1,\ldots, \ell\bigr\}$
  and $m\gg   
  0$, then   there exists a polynomial $P_m\in\kk[x]$ of degree $m$ such that
  $P(a_i)=0$ for any $i=1,\ldots,\ell$. Therefore,  $\aut_\kk(\A^2,\Delta)$
  contains the automorphisms $\bigl(x,y+P_m(x)\bigr)\in\pol$ for all
  $m\gg 0$, and it follows from  \cite[Theorem 1.3]{Ka} that $\aut_\kk(\A^2,\Delta)$
is not algebraic.
\end{rem}

  \begin{thm}
    \label{thm:blancprin}
Let   $\mathcal C\subset \A^2$ be a geometrically
irreducible and reduced  curve. Then, up to applying a plane automorphism,
$\mathcal C$ is one of the curves of the following list: 

    \beginenum

    \item $\mathcal C=\mathcal V\bigl(x^b-\lambda y^a  \bigr)$, where
      $a,b>1$ are coprime integers and $\lambda\in \kk^*$. If this is
      the case, then
      \[
\aut\bigl(\A^2,\mathcal C(\ack)\bigr) = \bigl\{ (t^ax,t^by)\mathrel{:}
t\in\kk^*\bigr\}\cong \kk^*.
        \]

\item $\mathcal C=\mathcal V\bigl(x^b y^a -\lambda  \bigr)$, where
      $a,b\geq 1$ are coprime integers, with $ab\neq 1$, and $\lambda\in \kk^*$. If this is
      the case, then
      \[
\aut\bigl(\A^2,\mathcal C(\ack)\bigr) = \bigl\{ (t^ax,t^{-b}y)\mathrel{:}
t\in\kk^*\bigr\}\cong \kk^*.
        \]

\item $\mathcal C=\mathcal V\bigl(xy -\lambda  \bigr)$, where
       $\lambda\in \kk^*$. If this is
      the case, then
      \[
\aut\bigl(\A^2,\mathcal C(\ack)\bigr) = \bigl\{ (tx,t^{-1}y)\mathrel{:}
t\in\kk^*\bigr\}\rtimes \{\operatorname{id},\sigma\}\cong \kk^*\rtimes \Z_2,
        \]
where $\sigma(x,y)=(y,x)$ is the permutation of coordinates.

\item   $\mathcal C=\mathcal V\bigl(\lambda
  x^2+\nu y^2-1\bigr)$, where $\car(\kk)\neq 2$, and   $\lambda,\nu\in\kk^*$  are such that 
  $-\lambda\nu$ is not a square in $\kk$. If this is
      the case, then $\aut\bigl(\A^2,\mathcal C(\ack)\bigr)$
       is the subgroup of 
       $\operatorname{GL}_2(\kk)$ preserving the form $\lambda
  x^2+\nu y^2$ : 
      \[
\aut\bigl(\A^2,\mathcal C(\ack)\bigr) = \bigl\{\left(\begin{smallmatrix} a& -\nu b\\
      \lambda b& a\end{smallmatrix}\right)
  \mathrel{:}
   (a,b)\in\kk\times\kk\,,\ 
   a^2+\lambda\nu b^2=1\bigr\}\rtimes
 \left\{\operatorname{id},\tau \right\}
=T_{\lambda,\nu}\rtimes \Z_2,
        \]
where $\tau(x,y)=(x,-y)$. Notice that $T_{\lambda,\nu}$ is a non-$\kk$-split torus.

\item $\mathcal C=\mathcal V\bigl(x^2+\mu xy+
  y^2-1\bigr)$, where  $\car(\kk)= 2$ and    $\mu\in\Bbbk^*$
                 is such that  $x^2+\mu
                 x+1$ has no root in $\kk$.  If this is
      the case, then $\aut\bigl(\A^2,\mathcal C(\ack)\bigr)$
       is the subgroup of 
       $\operatorname{GL}_2(\kk)$ preserving the form $x^2+\mu xy+
  y^2$: 
  \[
    \aut\bigl(\A^2,\mathcal C(\ack)\bigr) = \left\{
                 \left(\begin{smallmatrix} a&  b\\
                      b& a+\mu b\end{smallmatrix}\right)\mathrel{:}
a,b\in \kk\,, a^2+\mu ab + b^2=1\right\}\rtimes  \left\{\operatorname{id},\sigma_\mu\right\} =T_\mu\rtimes
              \Z_2,
            \]
where $\sigma_\mu(x,y)=(x+ \mu y,y)$.
Notice that $T_{\mu}$ is a non-$\kk$-split torus.

\item $\mathcal C$ is the line of equation $x=0$, in which case
  \[
  \aut\bigl(\A^2,\mathcal C(\ack)\bigr)=\Bigl\{
  \bigl(ax,by+P(x)\bigr) \mathrel{:} a,b\in\kk^*\,,\ P\in\kk[x]\Bigr\}.
\]

In particular, if $R:\aut\bigl(\A^2, \mathcal C(\ack)\bigr)\to
\aut(\A^1)=\bigl\{ y\mapsto by+c\mathrel{:} b\in\kk^*\,,
c\in\kk\bigr\}$ denotes the restriction map, then 
$
  \operatorname{Ker}(R)=
\Bigl\{(x,y) \mapsto \bigl(ax,y+P(x)\bigr) \mathrel{:} a\in\kk^*\,,\
P\in\kk[x]\,,\ P(0)=0\Bigr\}$
and

\[
    \aut\bigl(\A^2,\mathcal
    C(\ack)\bigr)=\operatorname{Ker}(R)\rtimes \aut(\A^1).
\]

 \item  In any other case, the curve $\mathcal C$ is  such that $\aut\bigl(\A^2,\mathcal C(\ack)\bigr)$ is finite.
                 
         \end{enumerate}
\end{thm}
\begin{proof}
  See \cite[Theorem 2]{BlSt}.
  \end{proof}
\begin{defi}
 We will abuse notations and  say that
 $D\subset \A^2(\ack)$  is a \emph{curve of type} (1) -- (7), if $D=\mathcal
 C(\ack)$, where  $\mathcal C$ verifies the  respective condition
 in  Theorem \ref{thm:blancprin}. The \emph{canonical form} of a  curve
   of type (1) -- (6) is the curve of equations given in the Theorem
   \ref{thm:blancprin} for the corresponding type.

Following \cite{BlSt}, we say that a curve $\mathcal C$ is of
\emph{fence type} if up to applying an automorphism of $\aut(\A^2)$,
$\mathcal C$ has equation $P(x)=0$, with $P\in\kk[x]$.  
   
  \end{defi}

\begin{cor}
    Assume that $\Delta$ is a proper closed subset of $\A^2(\ack)$. Then one of the following assertions holds:

    $(a)$  $\Delta$  is contained in a  curve of fence type, in which case $\aut(\A^2,\Delta)$ is not an algebraic group. 

    $(b)$ The group  $\aut(\A^2,\Delta)$ is conjugate to an algebraic subgroup of either ${\rm Aff}(\A^2)$ or either $J_n$ for some $n\geq 2$ (notations as in \cite{BlSt}). 
    \end{cor}
    \begin{proof}
      With the added hypothesis of the closedness of $\Delta$, the
      proof of  \cite[Proposition 3.11]{BlSt} now applies without obstructions.   
      \end{proof}

  \section{The orbit of a group of automorphisms of the plane}
\label{sec:orbitsubgruop}

Let $H\subset\aut(\A^2)$ be a subgroup and $p\in\A^2(\ack)$. We are
interested now in the calculation of
$\aut\bigl(\A^2,\orb_H(p)\bigr)$. Notice that  we have that
  \[
    H\subset \aut\bigl(\A^2,\orb_H(p)\bigr)\subset
    \aut\bigl(\A^2,\overline{\orb_H(p)}\bigr).
  \]

  \begin{lem}
    \label{lem:closorb}
Let $H\subset \aut(\A^2)$ be a subgroup and $p\in \A^2(\ack)$. Then
$\overline{\orb_H(p)} $ is either:

\beginenum
\item a finite set;

  \item a curve $\mathcal C=\mathcal C_1\cup\dots\cup \mathcal C_s$,
    with $\mathcal C_i\cong \mathcal C_j$ for all $i,j$, where
    $\mathcal C_1$ is a curve of type (1) -- (6);

  \item $\A^2(\ack)$.
    \end{enumerate}
    \end{lem}
    \begin{proof}
       Assume that
      $\overline{\orb_H(p)}$ is neither finite nor all the
      plane; then $\dim \overline{\orb_H(p)}=1$ and therefore  $H$ is
      necessarily infinite.   Let
      $\overline{\orb_H(p)}=\bigcup_{i=1}^s \mathcal C_i$ be the
      decomposition in irreducible components, with $\dim
      \mathcal C_1=1$.  By continuity and the
      irreducibility of $\mathcal C_1$, it follows that  there exists $h_i\in
      H$ such that  $h_i\cdot \mathcal C_1=\mathcal C_i$. Hence, it remains to prove
      that $\mathcal C_1$ is a curve of type (1) -- (6). 

Let $H_i=\{ h\in
H\mathrel{:} h\cdot  \mathcal C_1= \mathcal C_i\}\subset H$, $i=1,\dots
s$.  Since $H$ is infinite, there exists $i\in\{1,\dots,s\}$ such that $\#
H_i=\infty$. Let $h_i\in H_i$ be a fixed element. Then 
$h_i^{-1}H_i\subset H_1\subset \aut(\A^2,\mathcal C_1)$, and it follows
that $\#\aut(\A^2,\mathcal C_1)=\infty$. We deduce from Theorem  \ref{thm:blancprin}
that $\mathcal C_1$  is a curve of type (1) -- (6).
    \end{proof}

We collect now some technical remarks that we need in order
to calculate the stabilizers of an orbit.

\begin{lem}
  \label{lem:isotropia}
Let $\mathcal C$ be a curve of type (1) -- (6) given by its canonical
form and $p=(x_p,y_p)\in \mathcal C(\ack)$. If we keep the notations
of Theorem \ref{thm:blancprin}, then the isotropy group $\aut(\A^2,\mathcal C)_p$ is trivial unless

\beginenumr

\item $\mathcal C=\mathcal V\bigl(xy -\lambda  \bigr)$,  in which case
  $\aut\bigl(\A^2,\mathcal C(\ack)\bigr)_p=\bigl\{\Id, \frac{x_p}{y_p}\sigma\bigr\}$,
where $\sigma=(y,x)$ is the permutation of coordinates.

\item   $\mathcal C=\mathcal V\bigl(\lambda
  x^2+\nu y^2-1\bigr)$, in which case
  \[
    \aut\bigl(\A^2,\mathcal C(\ack)\bigr)_p=\bigl\{\Id,
  t_{\lambda x_p^2-\nu y_p^2,2x_py_p}\tau\bigr\}= \bigl\{\Id,
  t_{2\lambda  x_p^2- 1,2x_py_p}\tau\bigr\},
\]
where  $t_{a,b}=\left(\begin{smallmatrix} a& -\nu b\\
      \lambda b& a\end{smallmatrix}\right)\in T_{\lambda,\nu}$ and $\tau= \left(\begin{smallmatrix} 1& 0\\
     0& -1\end{smallmatrix}\right)$.
  
\item $\mathcal C=\mathcal V\bigl(x^2+\mu xy+
  y^2-1\bigr)$, in which case
\[
    \aut\bigl(\A^2,\mathcal C(\ack)\bigr)_p=\bigl\{\Id,
  t_{\mu x_py_p+1,\mu y_p^2 }\sigma_\mu\bigr\}= \bigl\{\Id,
  t_{x_p^2+y_p^2,\mu y_p^2}\sigma_\mu\bigr\},
\]
where  $t_{a,b}=\left(\begin{smallmatrix} a&  b\\
       b& a+\mu b\end{smallmatrix}\right)\in T_{\mu}$ and $\sigma_\mu= \left(\begin{smallmatrix} 1& \mu\\
     0& -1\end{smallmatrix}\right)$.

\item $\mathcal C$ is the line of equation $x=0$, in which case
  \[
    \begin{split}
    \aut\bigl(\A^2,\mathcal
    C(\ack)\bigr)_p= &\  \Bigl\{
      \bigl(ax,by+P(x)\bigr) \mathrel{:} a,b\in\kk^*\,,\ P\in\kk[x],
      P(0)=-by_p\Bigr\}\cong\\
    & \ \operatorname{Ker}(R)\rtimes \bigl\{ (0,y)\mapsto
    (0, b y+ (1-b)y_p)\mathrel{:}\ b\in \kk^*\bigr\}.
\end{split}
  \]

\end{enumerate}
\end{lem}
\begin{proof}
Let $G=\aut\bigl(\A^2,\mathcal C(\ack)\bigr)$.  If $\mathcal C$ is of
type (1) or (2), it is clear that $G_p=\{\Id\}$ for any $p\in \mathcal C(\ack)$.

\beginenumr
\item Easy calculations show that if $p=(x_p,y_p)$ is a $\ack$-point of $\mathcal V\bigl(xy -\lambda  \bigr)$, then $G_p=\bigl\{\Id,
\frac{x_p}{y_p}\sigma\bigr\}$.

\item Consider now  $\mathcal C=\mathcal V\bigl(\lambda
  x^2+\nu y^2-1\bigr)$ and let $p=(x_p,y_p)\in \mathcal C(\ack)$. Then
  $(T_{\lambda,\nu})_p=\{\Id\}$ and an involution $t_{a,b}\tau=\left(\begin{smallmatrix} a& \nu b\\
      \lambda b& -a\end{smallmatrix}\right)$ fixes $p$  if and only if
  there exists $s\in \kk^*$ such that
\begin{equation}
\label{eq:uno}
\left\{
\begin{array}{rcl}
-s\nu b &=& x_p\\
s(a-1)&=& y_p\\
\end{array}
\right.
\end{equation}

Indeed, $(-\nu b,a-1)$ is a basis of the space of  eigenvectors of eigenvalue
$1$ of $t_{a,b}\tau$.

In order to solve Equation \eqref{eq:uno}, we first recall that we
also have the compatibility conditions 
\begin{equation}
\label{eq:doso}
\left\{
\begin{array}{rcl}
a^2+\lambda\nu b^2&=& 1\\
\lambda x_p^2+\nu y_p^2&=& 1\\
\end{array}
\right.
\end{equation}

If $(a,b)\neq (1,0)$ then we can eliminate $s$ in Equation
\eqref{eq:uno}. If we substitute $x_p$ in Equation \eqref{eq:doso}, we
obtain that $(a,b)=(\lambda x_p^2-\nu y_p^2,2x_py_p) = (2\lambda  x_p^2-
1,2x_py_p)$. 
Notice that since $a\neq 1$, we deduce that $p\neq
\bigl(\pm\frac{1}{\sqrt{\lambda}},0\bigr)$. An easy calculations show
that $G_{_{( \pm\frac{1}{\sqrt{\lambda}},0)}}=\{\Id,\tau\}$.

\item If  $\mathcal C=\mathcal V\bigl(\lambda
  x^2+\nu y^2-1\bigr)$  and $p=(x_p,y_p)\in \mathcal
  C(\ack)$,  we proceed as in the previous case: we have  that 
  $(T_{\mu})_p=\{\Id\}$ and 
  $t_{a,b}\sigma_\mu=\left(\begin{smallmatrix} a& \mu a +b\\
       b& a\end{smallmatrix}\right)$ fixes $p$ if and only if
  there exists $s\in \kk^*$ such that
\[
\left\{
\begin{array}{rcl}
s(a-1) &=& x_p\\
sb&=& y_p\\
\end{array}
\right.
\]
with 
 compatibility conditions
\[
\left\{
\begin{array}{rcl}
a^2+\mu ab+ b^2&=& 1\\
 x_p^2+\mu x_py_p+  y_p^2&=& 1\\
\end{array}
\right.
\]

If $(a,b)\neq (1,0)$ then we can eliminate $s$ and obtain that
$(a,b)=(\mu x_py_p+1,\mu y_p^2) = (x_p^2+y_p^2,\mu y_p^2)$. 
Notice that $t_{1,\mu}\sigma_\mu$ has not $1$ as eigenvalue, and that since $a\neq 1$, we deduce that $p\neq
\bigl(1,0\bigr)$. An easy calculation show
that $G_{(1,0)}=\{\Id,\sigma_\mu\}$.

\item For the case  $\mathcal C=\mathcal V(x)$, just recall that
  $\aut\bigl(\A^2,\mathcal C(\ack)\bigr)$ is the semi-direct product of
  $\operatorname{Ker}(R)$ and $\aut(\A^1)$. 
\end{enumerate}
\end{proof}

\begin{rem}\label{rem:Gsimple}
 (1)   Let  $\mathcal C$ be  a curve of type (3) -- (5) and consider
  $G=\aut(\A^2,\mathcal C)$. Let  $p\in\mathcal C$ be such that
  $\overline{\orb_G(p)}=\mathcal C$ --- it is clear that such a $p$ exists. Then, it follows from Theorem
\ref{thm:blancprin} and   Lemma \ref{lem:isotropia} that
\[
G  =T\rtimes G_p= T\rtimes \{\Id,\tau_p\}=T\cup
  T\tau_p,
\]
where $T$ is a possibly non-split torus and  $\tau_p$ is the involution given in Lemma
\ref{lem:isotropia}.

Notice that $t\tau_p=t^{-1}\tau_p$ for all $t\in
T$.
In particular, any element of $T\tau_p$ is also an involution and
$\aut(\A^2,\mathcal C)=T\rtimes \{\Id,t\tau_p\}$ for all $t\in T$.

\noindent (2)
It follows that if $H\subset
\aut(\A^2,\mathcal C)$ is a subgroup, then either $H\subset T$ or
\[
  H= H_0\cup H_0\gamma = H_0\rtimes
  \{\Id, \gamma\},
\]
where $H_0=H\cap T$ and $\gamma\in H$ is a non trivial
involution. Indeed, if $\gamma, \gamma'\in H\cap (T\gamma)$, then
$\gamma'=h\gamma$, with $h\in H_0$. 

\noindent (3) In particular, if a subgroup $H\subset \aut(\A^2)$ 
and $p\in \A^2(\ack)$ are such that $\overline{\orb_H(p)}$  is a curve
of type (3) -- (5), it follows that either $H$ is a subgroup of the
corresponding torus $T\subset \aut\bigl(\A^2,\overline{\orb_H(0)}\bigr)$
  or
  \[
    H=H_0\rtimes \{\Id,t_0\tau_p\}=H_0\cup H_0t_0\tau_p,
    \]
    where $H_0=T\cap H_0$ and $t_0\in T$.

 Analogously, we have that $A=\aut\bigl(\A^2,\orb_H(p)\bigr)$ is either
 a subgroup of $T$ or
\[
   A =A_0\rtimes \{\Id,t_1\tau_p\}=A_0\cup A_0t_1\tau_p,
    \]
where $A_0=T\cap A_0$ and $t_1\in T$.

\end{rem}

  \begin{thm}
    \label{thm:autorbit}
  Let $H\subset \aut(\A^2)$ be a subgroup and let $p\in\A^2(\ack)$
  such that  $\mathcal C =\overline{ \orb_H(p)}$ is an
  irreducible curve. In the  notations  of Remark
  \ref{rem:Gsimple}, we have that:
  \beginenuma
\item  If $\mathcal C$ is a curve of type (1) or (2) then $A=H$.

  \item   If  $\mathcal C$ is a curve of type (3), (4) or
    (5),  then  $G_p=\{\Id,\tau_p\}\cong\Z_2$ and:
    \beginenumr
    \item If $H= H_0\subset T$ then $A= H_0\cup
      H_0\tau_p=H_0\rtimes G_p$.

      \item If $H=H_0\rtimes G_p=H_0\cup H_0\tau_p$, then
        $A=H$.

        \item If $H=H_0\rtimes
          \{\Id,t_0\tau_p\} = H_0\cup H_0t_0\tau_p$,  with $t_0\notin
          H_0$, then 
          $A=H$ unless $t_0^2\in H_0$, in which case
              \[
      A=\bigl\langle H_0\cup \{t_0\}\bigr\rangle \cup \bigr\langle H_0\cup
      \{t_0\}\bigr\rangle t_0\tau_p = \bigl\langle H_0\cup \{t_0\}\bigr\rangle 
      \rtimes \{\Id,t_0\tau_p\};
    \]

in particular, $A_0=\bigl\langle H_0\cup \{t_0\}\bigr\rangle$ and 
    $A_0/H_0\cong \Z_2$.
\end{enumerate}  

\item If $\mathcal C$ is a curve of type (6),
   then
  \[
  A=H (G_p\cap A).
\]
In particular, $H\operatorname{Ker(R)}\subset
A$.

\end{enumerate}
  \end{thm}
  \begin{proof}
    In what follows, we
      keep the notations of 
Theorem \ref{thm:blancprin} and  assume (without loss of generality) that $\mathcal C$ is given by
    its canonical form.  It is
    clear that $H\subset A\subset G$  --- the last inclusion follows by
    continuity of the action; then 
    $
      \mathcal C=\overline{\orb_H(p)}=
      \overline{\orb_A(p)}=\overline{\orb_G(p)}$.

    \beginenuma
  \item If $\mathcal C=\overline{\orb_H(p)}$ is of type (1) or (2), it
    easily follows that  $G_p=\{\Id\}$. Thus, if  $\varphi\in A$ then
    $\varphi(p)\in\orb_H(p)$, and therefore there exists $h\in H$ such
    that $h\varphi\in G_p$. It follows that $A=H$.

\item  Let
  $G=T\rtimes \{Id,\tau_p\}$; since $H\subset  A\subset G$, it follows
  from Remark \ref{rem:Gsimple} that both $H$ 
  and $A$ are determined by  their intersections with $T$ and the
   involutions   $t_0\tau_p$, with $t_0\in T$,
  belonging to the group. Hence,  we may  distinguish three cases:

\beginenumr

\item  $H= H_0$. If $s\in A_0$ then $sp\in \orb_H(p)$ and therefore
  $sp=hp$ for some
  $h\in H$, so   $s=h\in H_0$. On the other hand $\tau_p(hp)=
  h^{-1}\tau_pp=h^{-1} p \in\orb_H(p)$ for any $h\in H$. It follows that $A=H\rtimes
  \{\Id, \tau_p\}$.

\item $H=H_0\cup H_0\tau_p=H_0\rtimes G_p$. If
        $s\in A_0$ then  $sp = hp$ for some $h\in H$. If $h\in H_0$ we
        deduce that $s=h\in H_0$. If $h=h_0\tau_p$,
        then $hp=h_0p$ and it follows that $s\in H_0$. Hence 
        $A=H$.

  \item $H=H_0\cup H_0t_0\tau_p= H_0\rtimes\{\Id,t_0\tau_p\}$, where
    $t_0\notin H_0$.

    If  $s\in A_0$, then  $sp = hp$ for some $h\in H$. If $h\in H_0$ we
        deduce that $s=h\in H_0$. If $h=h_0t_0\tau_p$,
        then $hp=h_0t_0p$ and it follows that $s=h_0t_0$.
 Hence, $s\in A_0$ if and only if
$t_0\in A_0$ (because $h_0\in H_0$), and  it remains to  determine
under which conditions we have that 
 $t_0\in A_0$.

Consider $h_1\in H_0$; then
$t_0\cdot(h_1p) =h_1t_0p=h_1t_0\tau_p(p)\in\orb_H(p)$  (since $t_0\tau_p\in
H$) and
$t_0\cdot\bigl(h_1t_0\tau_p(p)\bigr)=h_1t_0^2\tau_p(p)=h_1t_0^2p$. Notice
that 
$h_1t_0^2p\in\orb_H(p)$ if and only if  either 
$h_1t_0^2p\in \orb_{H_0}(p)$, or there exists $h_2\in H_0$ such that
$h_1t_0^2p=h_2t_0\tau_p(p)=h_2t_0p$. In the first case
it follows that  $t_0^2\in H_0$, and in the last case it follows that
$h_1t_0^2=h_2t_0$, and therefore $t_0\in H_0$.  Since
$t_0h_1t_0\tau_p(p)= h_1t_0^2p$, we conclude  that  $t_0\in A_0$ if
and only if $t_0^2\in H_0$.
      \end{enumerate}

\item
    If $\mathcal C$ is the line of equation $x=0$ and $\varphi=
    \bigl(ax,by+P(x)\bigr)\in  
  A$, then $\bigr(0,by_p+ P(0)\bigr)=h\cdot p$ for some $h\in H$. It
  follows that $h^{-1}\varphi\in G_p$ and therefore 
  $\varphi\in HG_p$; the result 
  follows. 
     \end{enumerate}
    \end{proof}

    If $\overline{\orb_H(p)}$ is an irreducible curve of type (6), the
    inclusion $H\operatorname{Ker}(R)\subset A$ may be strict, as the
    following example shows.
    
\begin{exa}\label{exa:larecta}
Assume that $\operatorname{char}(\K)=0$ and let
$\varphi=(x,y+c)$. If  $H={\langle
  \varphi\rangle}=\{\varphi^n\mathrel{:} n\in \Z\}$, then 
$\psi=(ax, ny)$ belongs to $\aut\bigl(\A^2,\orb_H(0,0)\bigr)$ for
all $n\in \Z$ -- indeed $\varphi^n(0,0)=(0, nc)$ for all $n\in \Z$.
\end{exa}

\begin{cor}\label{coro:notalg}
Let $H\subset \aut(\A^2)$ be an infinite countable subgroup of plane
automorphisms and $p\in\A^2(\ack)$ a point such that $\mathcal C
=\overline{ \orb_H(p)}$ is an   irreducible curve. Then the stabilizer
of $\orb_H(p)$ is not an algebraic group.
\end{cor}
\begin{proof}
It follows from Theorem \ref{thm:autorbit} that if $\mathcal C$ is of
type (1)--(5) then $A$ is infinite countable and therefore  is not
algebraic. If $\mathcal C$ is of type (6) then  $A$  contains 
$\operatorname{Ker}(R)$ and therefore $A$ contains automorphisms of
arbitrary degree; it follows from \cite[Theorem 1.3]{Ka} that $A$
is not algebraic.
\end{proof}

  \section{The orbit of an automorphism of the plane}
\label{sec:orbitautom}

Let $\varphi\in
\aut(\A^2)$ and $p\in \A^2(\ack)$. 
In this section, as an application of
Theorem \ref{thm:autorbit}, we 
study the geometry  of the \emph{orbit of $p$ by $\varphi$}.  More precisely, we
consider $H=\langle \varphi\rangle$ and describe
$\mathcal C=\overline{\orb_H(p)}$, $\aut(\A^2,\mathcal C)$  and
$\aut\bigr(\A^2,\orb_H(p)\bigr)$.
Along this section, we keep the previous notations and  set $\orb_\varphi(p)=\orb_{H}(p)$.

 As Lemma \ref{lem:closorb} indicates, the closure of the orbit
 $\orb_\varphi(p)$ is either finite, a curve or the whole plane.
We begin by showing that under mild conditions, there exist pairs
$p\in\A^2(\ack)$ and $\varphi\in\aut(\A^2)$ such that
$\orb_\varphi(p)$ is infinite and $\mathcal C=\overline{\orb_\varphi(p)}$ is
of type (1)--(6) --- in particular,
  $\aut\bigr(\A^2,\orb_H(p)\bigr)$ is not algebraic (see Corollary \ref{coro:notalg}).

      \begin{exa}\label{exa3.1}
        \beginenuma

        \item Assume that $\kk$ is such that it contains an element
        $t_0\in\kk$ 
        that is not a root of unity.
        Then:

\beginenumr        
        \item If $p=(x_p,y_p)$, $x_py_p\neq 0$ and
        $\varphi(x,y)=(t_0^ax,t_0^by)$, with $a,b> 1$  coprimes,
        then $\mathcal C=\mathcal V(x^b - \frac{x_p^b}{y_p^a}y^a)$.

\item If $p=(x_p,y_p)$, $x_py_p\neq 0$ and
        $\varphi(x,y)=(t_0^ax,t_0^{-b}y)$, with $a,b\geq 1$  coprimes,
        then $\mathcal C=\mathcal V(x^by^a - x_p^by_p^a)$. Notice that
        if $ab\neq 1$ then $\mathcal C$ is of type (2), whereas
        if $a=b=1$, then $\mathcal C$ is of type (3).

        \item If $p=(x_p,y_p)\neq(0,0)$,  and
        $\varphi(x,y)=t_0\cdot(x,y)=(t_0x,t_0y)$, then $\mathcal C$ is a line.
\end{enumerate}

\item Assume now that  $\K$ is uncountable.

  \beginenumr

  \item If   $\car(\kk)\neq 2$  and there exist
    $\lambda,\nu\in\kk^*$  are such that   
    $-\lambda\nu$ is not a square in $\kk$, consider
       $T_{\lambda,\nu}=\bigl\{\left(\begin{smallmatrix} a& -\nu b\\
      \lambda b& a\end{smallmatrix}\right)
  \mathrel{:}
   (a,b)\in\kk\times\kk\,,\ 
   a^2+\lambda\nu b^2=1\bigr\}$. If $\varphi = \left(\begin{smallmatrix} a& -\nu b\\
      \lambda b& a\end{smallmatrix}\right)\in T_{\lambda,\nu}$, then
    $\varphi$  has eigenvalues $a\pm \sqrt{-\lambda\mu} b$. Since $\K$
    is uncountable,  there exists $(a,b)\in\K^2$ such that $a\pm
    \sqrt{-\lambda\mu} b$ is not a root of unity, and therefore
    $(0,0)$ is the unique periodic point of $\varphi$. It follows that
    if $p= (x_p,y_p)\in  \mathcal  C=  \mathcal V(\lambda x^2+\nu y^2-1)$, then $\orb_\varphi(p)$ is infinite, with
    $\overline{\orb_\varphi(p)}= \mathcal C$;
    in particular, $\overline{\orb_\varphi(p)}$ is of type (4).

    \item  Analogously, if  $\car(\kk)= 2$ and    $\mu\in\Bbbk^*$
                 is such that  $x^2+\mu
                 x+1$ has no root in $\kk$, one can prove that there
                 exists $\varphi\in  \left\{\left(\begin{smallmatrix} a&  b\\
                      b& a+\mu b\end{smallmatrix}\right)\mathrel{:}
a,b\in \kk\,, a^2+\mu ab + b^2=1\right\}$ such that $(0,0)$ is
the unique periodic point of $\varphi$, and therefore if $p\in \mathcal
C= \mathcal
V( x^2+\mu xy + y^2-1)$ then $\orb_\varphi(p)$ is infinite, with
    $\overline{\orb_\varphi(p)}= \mathcal C$;
    in particular, $\overline{\orb_\varphi(p)}$ is of type (5).
    \end{enumerate}
\end{enumerate}        
        \end{exa}

         If $\varphi\in\aut(\A^2)$, let $\mathcal S$ be
  the family of  $\varphi$-stable curves  and $\mathcal P$ the set of
  periodic points of $\varphi$. It is clear 
  that there exists
  $p\in\A^2(\ack)$ such that $\overline{\orb_\varphi(p)}=\A^2$
  if and only if $\mathcal S\cup\mathcal P\neq
  \A^2$. This simple remark allows us to show   that there exist
  automorphisms of the plane 
  with orbits which closure is all the plane as follows.

        \begin{exa}
   Let $\varphi\in\aut(\A^2)$; for $n\geq 1$, we denote by   $F_n$ the (closed)
        subset of fixed points of $\varphi^n$. Clearly, $\mathcal P=
        \bigcup F_n$.

           Recall that the 
 (first) \emph{dynamical degree} of $\varphi$
          is defined  as
        $\lambda_1(\varphi)=\lim_{m\to\infty} \sqrt[m]{\deg \varphi^m}$ --- notice that $\lambda_1(\varphi)>1$ implies  $\lambda_1(\varphi^n)>1$
for any $n\geq 1$; for more details see \cite{DF}, where the complex case is treated. It follows from  \cite[Thm. B]{Ca} and
\cite[Thm. 6.5]{Xi} that  if $\lambda_1(\varphi)>1$,  then $\mathcal
S$ is a curve and $F_n$ is finite. Therefore
$\mathcal S\cup \mathcal P$ is contained in a countable union of curves.
Hence, if $\ack$ is uncountable and $\lambda_1(\varphi)>1$, there
exists $p\in\A^2(\ack)$ such that $\overline{\orb_\varphi(p)}=\A^2$.

One particular example of this situation concerns the so-called H\'enon
maps, that is morphisms  of the form $\varphi= \bigl(y,-\delta
x+P(y)\bigr)$, where $\delta\in\K^*$ and $P\in\K[y]$ is a polynomial of
degree greater or equal that $2$, since in this case $\lambda_1(\varphi)=\deg P$.
\end{exa}

\emph{A priori}, the closure of an orbit $\orb_\varphi(p)$ is not
necessarily an irreducible curve. In what follows, we exploit Lemma \ref{lem:pro_gal} in order to describe
$\widehat{\orb_\varphi(p)}$. 
  If $\orb_\varphi(p)$ is
  finite then $\widehat{\orb_\varphi(p)}=
  \operatorname{Gal}(\ack/\kk)\cdot \orb_\varphi(p) $.
  On the other hand, if $\overline{\orb_\varphi(p)}=\A^2(\ack)$, then
  $\widehat{\orb_\varphi(p)}$ is the whole plane.
  Next proposition deal with the case when
  $\overline{\orb_\varphi(p)}$ is a curve.

  \begin{pro}\label{pro3.2}
    Let  $\varphi\in\aut(\A^2)$ and $p\in\A^2_{\ack}$  be such that
    $\overline{\orb_\varphi(p)}=\mathcal C$ is a curve, and let
    $\mathcal C=\bigcup_{i=1}^\ell\mathcal C_i$ and
    $\widehat{\orb_\varphi(p)} = \bigcup_{i=1}^s\mathcal C_i$ be the decomposition in
    irreducible components of $\mathcal C$ and
    $\widehat{\orb_\varphi(p)}$ respectively (see Lemma
    \ref{lem:pro_gal}). Then $s=k\ell$ and,  up to reordering the indexes, we
    have that:
    
 \beginenum
    \item The morphism $\varphi$ restricts to isomorphisms
      $\varphi|_{_{\mathcal C_{i+j\ell}}}:\mathcal C_{i+j\ell}\to \mathcal
      C_{i+j\ell+1}$ and
      $\varphi|_{_{\mathcal C_{j\ell +\ell}}}:\mathcal C_{j\ell+\ell}\to \mathcal
      C_{j\ell+1}$, for $i=1,\ldots,\ell-1$, $j=0,\dots, k-1$.

    \item  $\mathcal C_{i+j\ell}=\overline{\orb_{\varphi^{\ell}}\bigl(\varphi^{i+j\ell-1}(p)\bigr)}$
      for $i=1,\ldots, \ell$,  $j=0,\dots, k-1$.

      \item  The curves $\mathcal C_i$ are isomorphic --- in
        particular, the curves have the same  type.
      \end{enumerate}
    \end{pro}

    \begin{proof}
     Without loss of generality we
      assume that $p\in \mathcal C_1$. We reorder $\mathcal
      C_1,\dots, \mathcal C_\ell$ in such a way that $\mathcal C_i=
      \varphi^{i-1}(\mathcal C_1)$ for $i=1,\dots, \ell$ (see
      Lemma \ref{lem:closorb}); then   $\mathcal C_1=\varphi(\mathcal
      C_\ell)$.    It follows from Galois descent that 
      \[
        \widehat{\orb_\varphi(p)}=\operatorname{Gal}(\ack/\kk)\cdot 
     \mathcal C= \bigcup_{i=1}^\ell  \operatorname{Gal}(\ack/\kk)\cdot
     \mathcal C_i,
\]
where if $i\neq j$, then $ \operatorname{Gal}(\ack/\kk)\cdot
     \mathcal C_i\cap  \operatorname{Gal}(\ack/\kk)\cdot
     \mathcal C_j$ does not contain a curve. Since  $\varphi$ is fixed
     by the 
      Galois action, it follows from Galois descent that  if $i=1,\dots 
      ,\ell-1$, then 
      $\varphi|_{\mathcal C_i}: \mathcal C_i\to \mathcal C_{i+1}$
      induces an $\kk$-isomorphism  $\operatorname{Gal}(\ack/\kk)\cdot
     \mathcal C_i\to \operatorname{Gal}(\ack/\kk)\cdot
     \mathcal C_{i+1}$. Assertion (1) follows; assertion (2) is clear.

In order to prove assertion (3), we first observe that $\mathcal
C_{j\ell+1},\dots, \mathcal C_{j\ell+\ell}$ being isomorphic for
$j=0,\dots k-1$, they have the same type. Up to reordering, we can
assume that there exists   $g_j\in\operatorname{Gal}(\ack/\kk)$, $j=1,\dots, k-1$, such that     $\mathcal C_{i+j\ell} =g_j\cdot \mathcal C_{i}$.
 Since    $\psi\in
\aut(\A^2,g_j\mathcal C_1) $ if only if $\psi = g^{-1}\psi g\in
\aut(\A^2,\mathcal C_1)$, the result follows.
    \end{proof}

      \begin{exa}\label{masejem}
       Suppose that $\Delta\subset\A^2(\ack)$ is such that $\widehat{\Delta}$ is a  fence of
       equation  $F(x)=0$, with $F\in\kk[x]$ --- then it follows from
       \cite[Theorem 1]{BlSt}  that       
        \[
\aut\bigl(\A^2,\widehat{\Delta}\bigr)=\bigl\{\bigl(\alpha x+\beta,\gamma
y+P(x)\bigr)\mathrel{:}
\alpha,\beta,\gamma\in\kk\,,\  F(\alpha x+\beta)/F(x)\in\kk^*\,, \ P\in\kk[x]\bigr\}. 
\]
     
\beginenuma
   \item  If $F=x-1$, then $\varphi=\bigl(\alpha x+\beta,\gamma
y+P(x)\bigr)\in \aut\bigl(\A^2,\widehat{\Delta}) $  if and only if $\alpha=1$ and
$\beta=0$. Choose $\gamma\in\kk^*$ which is not a root
     of the unity and consider   $p=(1,1)\in\A^2(\ack)$.  Let
     $P\in\kk[x]$ be such that $P(1)=1$. Then $\varphi^n(p)\neq p$
     for all $n$ and  $\ophi(p)=\bigl\{(1,\sum_{i=0}^n \gamma^i);
     n\in\Z\bigr\}$. It follows that
     $\overline{\ophi(p)}=\widehat{\ophi(p)}_{\ack}=\widehat{\Delta}$.  Thus, in this
     case $\widehat{\ophi(p)}_{\ack}$ is an irreducible curve, stable
     by $\varphi$ --- in the notations of Proposition \ref{pro3.2},
     $\ell=s=1$.  Recall  that in this case the group
     $\aut\bigl(\A^2,\ophi(p)\bigr)$       is not
     algebraic.

  \item  If $F=x^2-1$ , then $\varphi=\bigl(\alpha x+\beta,\gamma
y+P(x)\bigr)\in \aut\bigl(\A^2,\widehat{\Delta}\bigr) $  if and only if
$\alpha=\pm 1$ and 
$\beta=0$. Choose $\gamma\in\kk^*$ which is not a root
     of the unity,  $d\geq 1$,  and consider the automorphism
     $\varphi_d=\bigl(-x,\gamma y+(x^2-1)^d\bigr)$. If $p=(1,y_0)$, with $y_0\neq 0$, then
     $\orb_{\varphi_d}(p)=\bigl\{\bigl((-1)^n,\gamma^n y_0\bigr); n\in\Z\bigr\}$ is dense
     in the fence $F(x)=0$. It follows that
     $\widehat{\ophi(p)}_{\ack}=\overline{\ophi(p)}=
     \widehat{\Delta}$, and therefore it is a curve with two
     irreducible components. Moreover, in the notations of Proposition
     \ref{pro3.2}, $\ell=s=2$ and once again
     $\aut\bigl(\A^2,\ophi(p)\bigr)$ is not algebraic because the orbit does not
     depend on $d=\deg \varphi_d$. 
\end{enumerate}
 \end{exa}

 The following example, very similar to the ones given above shows
 that is not sufficient for $\Delta$ to be  an orbit in order to
 guarantee  $\overline{\Delta}$ and $\widehat{\Delta}$ to be equals.

 \begin{exa}\label{masejem2}
 Assume that $\K\neq \ack$ and  let
 $\xi\in\ack\setminus\K$. Denote  by $P\in\K[x]$ the minimal polynomial
 of $\xi$ over $\K$ and consider   $\varphi=\bigl(x,\gamma
 y+P(x)\bigr)\in\aut(\A^2)$, where $\gamma\in\kk^*$ is not a root of
 unity. If $p=(\xi,1)$ , then $\Delta=\orb_\varphi(p)=\bigl\{(\xi,\gamma^n); n\in\Z\bigr\}$
 is infinite. Clearly $\overline{\Delta}=\mathcal V_{\ack}(x-\xi)$ and
 $\widehat{\Delta}=\mathcal V_{\ack}(P)$ so
 $\Delta\subsetneq  \overline{\Delta}\subsetneq \widehat{\Delta}$. 

Notice that the automorphisms $(x,y)\mapsto \bigl(x,y+P(x)^d\bigr)$,  $d\geq 1$, 
 fix every point in $\Delta$ and therefore 
 they belong to $\aut(\A^2,\Delta)$.
\end{exa}

\begin{rem}
  \label{remnoirred}
Let        $\varphi\in\aut(\A^2)$ and   $p=(x_p,y_p)\in \A^2\bigl(\ack\bigr)$ be such  that 
         $\overline{\ophi(p)}=\mathcal C_1\cup\dots \cup \mathcal C_\ell$
         is a curve. Then, by Proposition \ref{pro3.2}, up to
         reordering we can assume that 
         $\mathcal C_i=\varphi^{i-1} \mathcal C_1$, for $i=0,\dots, \ell
         -1$, with $\varphi^\ell(\mathcal C_1)=\mathcal
         C_1$. Therefore
\[
\mathcal C_1=\overline{\orb_{\varphi^\ell}(p)}.
\]
         
Moreover,
         each curve $\mathcal C_i$  is of the same type and, up to
         conjugation by an automorphism of $\A^2$, we can
          assume that $\mathcal C_1$ is one of the list given in  
         Theorem \ref{thm:blancprin}.

         Since any automorphism   $\psi\in
         \aut\bigl(\A^2,\overline{\ophi(p)}\bigr)$ induces a
         permutation of the irreducible curves $\mathcal C_i$, we have
         the following exact sequence of abstract groups:
         \[
           \xymatrix{
             1\ar[r]& \bigcap_{i=1}^\ell\aut(\A^2,\mathcal C_i)\ar[r] &
             \aut\bigl(\A^2,\overline{\ophi(p)}\bigr)\ar[r]&
             \mathfrak{S}_{\ell}
                          }
                        \]
 where $\mathfrak{S}_\ell$ is group of permutations of $\ell$
 elements. Since $\varphi^{n\ell}\in
 \aut\bigl(\A^2,\overline{\orb_{\varphi^\ell}(\varphi^i(p))}\bigr)$, if follows that
 $\bigcap\aut(\A^2,\mathcal C_i)$ is an infinite subgroup of
 $\aut(\A^2,\mathcal C_1)$.

Moreover,  $\aut(\A^2,\mathcal C_i)=\varphi^{i-1}\aut(\A^2,\mathcal
C_1)\varphi^{1-i}$ and if $\psi\in
\aut\bigl(\A^2,\overline{\ophi(p)}\bigr)$, then 
 there exists $i\in\{0,\dots \ell-1\}$ such that $\varphi^i\psi\in
 \aut(\A^2,\mathcal C_1)$. It follows that  $\psi\in
 \varphi^{-i}\aut(\A^2,\mathcal C_1)$.
  \end{rem}

  \begin{nota}
If $G$ is a group and $A,B\subset G$ arbitrary subsets, we denote
$AB=\{ab\mathrel{:} a\in A\,,\ b\in B\}$. 
    \end{nota}

      \begin{thm}\label{thm_2} Let
        $\varphi\in\aut(\A^2)$ and   $p=(x_p,y_p)\in \A^2\bigl(\ack\bigr)$ be such  that 
         $\overline{\ophi(p)}=\mathcal C_1\cup\dots \cup \mathcal C_\ell$
         is a curve. Then  up to applying an automorphism of $\A^2$
         one of the following assertions holds:   

\beginenuma

\item  $\mathcal C_1$ is of type (1) or (2) and $p\neq(0,0)$. In this case,
  $\aut\bigl(\A^2,\ophi(p)\bigr)=  \langle \varphi\rangle$.

\item  $\mathcal C_1$ is of type (3), (4) or (5). If we denote 
  $\aut(\A^2,\mathcal C_1)_p=\{\Id,\tau_p\}$ (see Lemma
  \ref{lem:isotropia}),  we have
  two possibilities:

  \beginenumr
\item If $\tau_p\varphi=\varphi^i\tau_p$ for some $i\in \Z$, then
  \[
    \aut\bigl(\A^2,\ophi(p)\bigr)=\langle \varphi\rangle\rtimes
    \{\Id,\tau_p\}=\langle\varphi\rangle\cup
    \langle\varphi\rangle\tau_p
  \]

  \item In other case, $    \aut\bigl(\A^2,\ophi(p)\bigr)=\langle
    \varphi\rangle$. 

\end{enumerate}

\item $\mathcal C_1$ is of type (6). In this case 
  \[
    \begin{split}
      \aut\bigl(\A^2,\orb_\varphi(p)\bigr)
      = & \    \bigcap_{i=0}^{\ell-1} \langle\varphi\rangle \Bigl(\varphi^i\Bigl(G_p\cap \aut\bigl(\A^2,\orb_{\varphi^\ell}(p)\bigr)\Bigr)\varphi^{-i}\Bigr)\\
     = &\  \bigcap_{i=0}^{\ell-1} \langle\varphi\rangle \Bigl(\Bigl(G_p\cap \aut\bigl(\A^2,\orb_{\varphi^\ell}(p)\bigr)\Bigr)\varphi^{-i}\Bigr)\\
    =  & \ \bigcap_{i=0}^{\ell-1} \langle\varphi\rangle \Bigl(\aut\bigl(\A^2,\orb_{\varphi^\ell}(p)\bigr)\varphi^{-i}\Bigr)  , 
    \end{split}
  \]
where $G=\aut\bigl(\A^2,\mathcal C_1(\ack)\bigr)$.

Moreover, if $\kk$ is uncountable, there exist infinitely many
automorphisms $\varphi\in G$ such that
\[
  \aut\bigl(\A^2,\orb_\varphi(p)\bigr)= \bigcap_{i=0}^{\ell-1}
  \langle\varphi\rangle \bigl(\varphi^i \operatorname{Ker}(R)\varphi^{-i}\bigr)= \bigcap_{i=0}^{\ell-1}
  \langle\varphi\rangle \bigl(\operatorname{Ker}(R)\varphi^{-i}\bigr),
  \]
where $R:\aut\bigr(\A^2,\mathcal C_1(\ack)\bigr)\to
  \aut\bigl(\mathcal C_1(\ack)\bigr)$ is the restriction morphism (see \ref{thm:blancprin}).
\end{enumerate}
\end{thm}

      \begin{proof}
        By Proposition \ref{pro3.2}, we have that 
        $\{\varphi^n\}_{n\in\Z}$ is infinite and $\mathcal C_1$ is of type (1) -- (6).
        Hence, it remains  to 
        describe   $\aut\bigl(\A^2,\ophi(p)\bigr)$ in each case.  
We assume that  $\mathcal C_i=\varphi^{i-1}(\mathcal C_1)$, and
$\varphi^\ell\in  \aut(\A^2,\mathcal C_1)$ (see Remark \ref{remnoirred}).

\beginenuma
\item If $\mathcal C_1= \mathcal V(x^b-y^a)$ 
  then,  since  $\varphi^\ell\in \aut(\A^2,\mathcal C_1)$, we have
  that $\varphi^\ell(0,0)=(0,0)$ and therefore $p\neq (0,0)$.
  If $\psi\in
  \aut\bigl(\A^2,\ophi(p)\bigr)$, then by Remark 
  \ref{remnoirred} $\psi=\varphi^j t$ for some $t\in \aut(\A^2,\mathcal
  C_1)$ and some $j\in \Z$.  But if $n\in \Z$, then there exists $s_n$
  such that 
  $\varphi^jt\varphi^{n\ell}(p)=\varphi^{s_n}(p)$. If follows that
  $ \varphi^{s_n-j}(p)=t\varphi^{n\ell}(p)\in \mathcal C_1$, and
  therefore $s_n-j$ is a multiple of $\ell$ and $t\in
  \aut(\A^2,\orb_{\varphi^\ell}(p))$. We deduce from   Theorem
  \ref{thm:autorbit} that $t\in \langle \varphi^\ell\rangle$, and
  therefore
  $
  \aut\bigl(\A^2,\ophi(p)\bigr)=\langle \varphi \rangle$.

If $\mathcal C_1$ is of type (2) an analogous argument applies.

\item Since 
  $\overline{\orb_{\varphi^\ell}(p)}=\mathcal C_1$ and
  $\varphi^\ell\in  \aut(\A^2,\mathcal C_1)$ then, in the
  notations of Theorem \ref{thm:autorbit}, there exists
  $t_0\in \K^*\subset \K^*\rtimes \{\Id,\tau_p\}=\aut(\A^2,\mathcal
  C_1)$ such that    $\varphi^\ell=t_0$. As in the previous case, if $\psi\in
  \aut\bigl(\A^2,\ophi(p)\bigr)$ we deduce that there exist $i\in \Z$
  and  $h_0\in  
  \aut\bigl(\A^2,\orb_{\varphi^\ell}(p)\bigr)$ such that
  $\psi=\varphi^ih_0$.  Applying Theorem \ref{thm:autorbit} to
  $H=\langle\varphi^\ell\rangle$, we deduce that $h_0\in \langle
  \varphi^\ell\rangle\rtimes\{\Id,\tau_p\}$.

If there exists $\psi\in \aut\bigl(\A^2,\ophi(p)\bigr)\setminus
\langle\varphi\rangle$, then 
$h_0\notin \langle\varphi^{\ell}\rangle$ and there exists
$n\in \Z$ such that 
$h_0=\varphi^{n\ell}\tau_p\in \aut\bigl(\A^2,\ophi(p)\bigr)$ ---
because $\varphi\in \aut\bigl(\A^2,\ophi(p)\bigr)$ ---, and it follows that
$\tau_p\in \aut\bigl(\A^2,\ophi(p)\bigr)$.

If $j\in \Z$, then $\tau_p\bigl(\varphi^j(p)\bigr)=\varphi^{i_j}(p)$
if and only if $\varphi^{i_j}\tau_p\varphi^j\in \aut(\A^2,\mathcal
C_1)_p$. Taking $i=-1$ we deduce that there exists $i\in \Z$ such that
 $\tau_p\varphi=\varphi^{i}\tau_p$ for some
$i\in \Z$ and the assertion is proved.

\item   
Assume now that $\mathcal C_1=\mathcal
    V(x)$. Since $\mathcal C_1\cap \mathcal C_i$ is finite and
    $\varphi^\ell$-stable for all $i=2,\dots,\ell$, it follows that
    $p\notin \mathcal C_1\cap \mathcal C_i$, because
    $\orb_{\varphi^\ell}(p)$ is infinite  --- more in general, $\mathcal C_i\cap\mathcal C_j\cap \orb_\varphi(p)=\emptyset$.
   Moreover, we deduce from Theorem
    \ref{thm:blancprin}  that
    $\varphi^\ell=\bigl(ax,by+P(x)\bigr)$, where $a,b\in \K^*$ and
    $P\in \K[x]$ --- notice that $b$ is not a root of unity different
    from $1$, since
    $\varphi^{n\ell}(p)= \bigr(0, b^ny_p+(1+\dots+b^{n-1})P(0)\bigr)$, and
    that if $\operatorname{char}(\K)\neq 0$, then also $b\neq 1$.

If $\psi\in \aut\bigl(\A^2,\orb_\varphi(p)\bigr)$, then
$\psi(p)=\varphi^i(p)$ for some $i$,  and therefore
$\varphi^{-i}\psi(p)=p \in \mathcal C_1$. Since $\varphi^{-i}\psi (\mathcal
C)=\mathcal C$, it follows that $\varphi^{-i}\psi (\mathcal C_1)=\mathcal
C_1$. Hence, in the notations of Theorem \ref{thm:blancprin} and Lemma
\ref{lem:isotropia}, 
\[
  \begin{split}
  \varphi^{-i}\psi\in G_p= &\ \Bigl\{
      \bigl(a'x,cy+S(x)\bigr) \mathrel{:} a',c\in\kk^*\,,\ S\in\kk[x],
      S(0)=-cy_p\Bigr\}\cong\\
      & \ 
  \operatorname{Ker}(R)\rtimes\bigl\{ (0,y)\mapsto (0,
  c y+(1-c)y_p)\mathrel{:}c\in\K^*\bigr\}
\end{split}
\]

Moreover, since $\orb_{\varphi^\ell}(p)=\orb_\varphi(p)\cap \mathcal C_1$,
  we deduce that  $\varphi^{-i}\psi \in
  \aut\bigl(\A^2,\orb_{\varphi^\ell}(p)\bigr)$ and therefore
  \[
    \psi\in \langle
  \varphi\rangle \bigl(G_{p}\cap
  \aut\bigl(\A^2, \orb_{\varphi^\ell}(p)\bigr)\bigr)= \bigl\{ \varphi^j\gamma\mathrel{:} j\in\Z\,,\ \gamma\in G_{p}\cap
  \aut\bigl(\A^2, \orb_{\varphi^\ell}(p)\bigr)\bigr\}.
\]

Analogously, since $\psi\bigl(\varphi^i(p)\bigr)=\varphi^{j_i}(p)$ for $i=1,\dots,\ell-1$, it follows that $\varphi^{i-j_i}\psi\in \aut(\A^2,\mathcal C_i)_{\varphi^i(p)}$. Hence 
\[
  \begin{split}
\psi \in & \langle\varphi\rangle \Bigl(\aut(\A^2,\mathcal
C_{i+1})_{\varphi^i(p)}\cap
\aut\bigl(\A^2,\orb_{\varphi^\ell}(\varphi^i(p))\bigr)\Bigr)=\\
& \langle\varphi\rangle \Bigl(\bigl(\varphi^iG\varphi^{-i}\bigl)_{\varphi^i(p)}\cap
\bigl(\varphi^i\aut\bigl(\A^2,\orb_{\varphi^\ell}(p)\bigr)\varphi^{-i}\bigr)\Bigr)=
\\
& \langle\varphi\rangle \Bigl(\varphi^iG_p\varphi^{-i}\cap
\bigl(\varphi^i\aut\bigl(\A^2,\orb_{\varphi^\ell}(p)\bigr)\varphi^{-i}\bigr)\Bigr)=
\\
& \langle\varphi\rangle \Bigl(\varphi^i\Bigl(G_p\cap \aut\bigl(\A^2,\orb_{\varphi^\ell}(p)\bigr)\Bigr)\varphi^{-i}\Bigr).
\end{split}
\]

  All in all, we have that
\[
  \psi\in  \bigcap_{i=0}^{\ell-1} \langle\varphi\rangle \Bigl(\varphi^i\Bigl(G_p\cap \aut\bigl(\A^2,\orb_{\varphi^\ell}(p)\bigr)\Bigr)\varphi^{-i}\Bigr)=\bigcap_{i=0}^{\ell-1} \langle\varphi\rangle \Bigl(\Bigl(G_p\cap \aut\bigl(\A^2,\orb_{\varphi^\ell}(p)\bigr)\Bigr)\varphi^{-i}\Bigr)
\]

 Let 
  $\psi\in \bigcap_{i=0}^{\ell-1} \langle\varphi\rangle
  \bigl(\varphi^i\bigl(G_p\cap
  \aut\bigl(\A^2,\orb_{\varphi^\ell}(p)\bigr)\bigr)\varphi^{-i}\bigr)$;
  in order to prove the first assertion of (c) we need to prove that   $\psi\in \aut\bigl(\A^2,\orb_\varphi(p)\bigr)$.
If $j\in \Z$, let $j=q\ell+r$ with $0\leq r<\ell$ and consider the decomposition $\psi= \varphi^r\gamma\varphi^{-r}$, with $\gamma\in G_p\cap
  \aut\bigl(\A^2,\orb_{\varphi^\ell}(p)\bigr)$. Then
\[
  \psi(\varphi^j(p))= \varphi^r\gamma\varphi^{-r+j}(p)=\varphi^r\gamma\varphi^{q\ell}(p).
  \]

  Since $\gamma\in \aut\bigl(\A^2,\orb_{\varphi^\ell}(p)\bigr)$, the
  assertion follows.

In order to prove the last assertion, we first observe that, up to conjugate with the translation of vector $(0,y_p)$,  we can assume
that $p=(0,0)$.  Let $\psi=\bigl(a'x, b'y+Q(x)\bigr)\in
\bigcap_{i=0}^{\ell-1}\varphi^{i} G_p\varphi^{-i}\setminus
\operatorname{Ker}(R)$. Then $Q(0)=0$ and,  since
$\psi\in \aut(\A^2,\orb_{\varphi^\ell}(p)\bigr)$,  there exists
$j\in\Z$ such that  
\[
\bigl(0,b'P(0)\bigr) = \psi \bigl(0,P(0)\bigr)=\psi \bigl(\varphi^\ell(p)\bigr)=\varphi^{j\ell}(p).
\]

If $j=1$, then $\psi=\bigl(a'x, y+Q(x)\bigr)\in
\operatorname{Ker}(R)$. Thus, up to work with $\psi^{-1}$ if necessary, we
can assume that $j>1$. We deduce that $b'=1+\dots+ b^{j-1}$; in
particular $\varphi$ is such that $1\neq 1+\dots+ b^{j-1}$  --- that
is, $b\neq 0$ is not a $(j-1)$-root of unity.

Analogously, we have that   $ \psi \bigl(0,
(1+b)P(0)\bigr)=\psi\bigl(\varphi^{2\ell}(p)\bigr) = \varphi^{m\ell}(p)$,
with $m\neq 2$.  Since $\varphi^{-\ell}=\bigl(a^{-1}x, b^{-1}y -
b^{-1}P(x)\bigr)$, we deduce that 
\[
  \bigr(0, (1+\dots+b^{j-1})(1+b)P(0)\bigr)=
  \psi\bigl(\varphi^{2\ell}(p)\bigr)=\begin{cases}
    \bigl(0,(1+\dots + b^{m-1})P(0)\bigr), \quad \mbox{if}\ m>2 \\
        \bigl(0,P(0)\bigr), \quad \mbox{if}\ m=1\\
          (0,0), \quad \mbox{if}\ m=0\\
   (0,-b^{-m}(1+\dots +b^m)P(0)), \quad \mbox{if}\ m< 0 \\
    \end{cases}
\]

For all four cases above we deduce a polynomial condition on
$b$. Since there are countably many such conditions, each of them with finite solutions, we
deduce that if $\K$ is uncountable there exist infinite values of $b$
for which no such a $\psi$ can be found. 
\end{enumerate}
\end{proof}

\begin{rem}
  Notice that  the proof of the last assertion of  Theorem \ref{thm_2}(c), we
  omitted additional conditions on $b$ that must be satisfied in order
  to allow the existence of $\psi\in
  \bigcap_{i=0}^{\ell-1}\varphi^iG_p\varphi^{-1}\setminus
  \operatorname{Ker}(R)$; namely, the conditions arising from the
  equations $\psi\bigl(\varphi^{j\ell}(p)\bigr)\in \orb_{\varphi^\ell}(p)$, for
  $j=3,\dots,\ell-1$. 
However,   Example \ref{exa:larecta} shows that there exists an
automorphism  $\varphi\in
  \aut\bigl(\A^2,\{x=0\}\bigr)$ such that $\aut\bigl(\A^2,
  \orb_\varphi(p)\bigr)\neq \langle \varphi\rangle
  \bigl(\bigcap_{i=0}^{\ell-1}\varphi^i\operatorname{Ker}(R)\varphi^{-i}\bigr)$, therefore, these 
   conditions may  be satisfied in particular examples.
  \end{rem}

\begin{rem}
In examples \ref{masejem} and \ref{masejem2} we exhibited cases where 
$\overline{\orb_\varphi(p)}$ is a fence and
$\aut\bigl(\A^2, \orb_\varphi(p)\bigr)$ is not algebraic. It is an
open question to 
prove that this is always the case  or to  provide an example 
where $\overline{\orb_\varphi(p)}$ is an union of curves of type (6)
but $\aut\bigl(\A^2, \orb_\varphi(p)\bigr)$ is  algebraic.

  \end{rem}

\end{document}